\documentclass[12pt]{amsart}
\usepackage{epsfig,color}

\makeatother

\theoremstyle{definition}

\def\fnum{equation} 
\newtheorem{Thm}[\fnum]{Theorem}
\newtheorem{Cor}[\fnum]{Corollary}

\newtheorem{Lem}[\fnum]{Lemma}

\newtheorem{Pro}[\fnum]{Proposition}

\numberwithin{equation}{section}
\newcommand{\nn}{{\bf{n}}}

\newcommand{\dist}{{\text {dist}}}

\newcommand{\Hess}{{\text {Hess}}}

\def\RR{{\bold R}}

\def\SS{{\bold S}}

\newcommand{\cS}{{\mathcal{S}}}

\newcommand{\eqr}[1]{(\ref{#1})}

\title[Level set flow]{Regularity of the level set flow}
\author[]{Tobias Holck Colding}%
\address{MIT, Dept. of Math.\\
77 Massachusetts Avenue, Cambridge, MA 02139-4307.}
\author[]{William P. Minicozzi II}%
 
\thanks{The  authors
were partially supported by NSF Grants DMS 1404540 and DMS 1206827.   This material is based upon work supported by the NSF DMS 1440140, while T.H.C. was in residence at the Mathematical Science Research Institute (MSRI) in Berkeley, CA, during the Spring of 2016.}

\email{colding@math.mit.edu  and minicozz@math.mit.edu}

\begin{document}

\maketitle

\begin{abstract}
We showed earlier that the level set function of a monotonic advancing front is twice differentiable everywhere with bounded second derivative.  We show here that  the second derivative is continuous if and only if  the flow has a single singular time where it becomes extinct and the singular set consists of a closed $C^1$ manifold with cylindrical singularities.

\end{abstract}

\section{Introduction}

 The level set method has been used with great success the last thirty years in both pure and applied mathematics to describe evolutions of various physical situations. In mean curvature flow, the evolving hypersurface (front) is thought of as the level set of a function that satisfies a nonlinear degenerate parabolic equation. Solutions are defined weakly  in the viscosity sense; in general, they may not even be differentiable (let alone twice differentiable).

For a monotonically advancing front, we showed in \cite{CM5} that viscosity solutions are in fact twice differentiable and satisfy the equation in the classical sense. Here we characterize when they are $C^2$.  As we will see, the situation becomes very rigid when the second derivative is continuous.

When $v: \RR^{n+1} \times \RR \to \RR$ is a function and for each $s$ the level set $t \to \{ x \, | \, v(x,t) = s \}$ evolves by the mean curvature flow, then $v$ satisfies the level set equation
 \begin{align}  \label{e:levelsetflow}
\partial_t  v=|\nabla v|\,\text{div}\left(\frac{\nabla v}{|\nabla v|}\right)\, .
\end{align}
This equation has been studied extensively.  Whereas the work of Osher and Sethian, \cite{OsSe}, was numerical, Evans and Spruck, \cite{ES}, and, independently, Chen, Giga, and Goto, \cite{ChGG} provided the theoretical justification. This is analytically subtle, principally because the mean curvature evolution equation is nonlinear, degenerate, and indeed  defined only weakly at points where $\nabla v = 0$. Moreover, $v$ is a priori not even differentiable, let alone twice differentiable. They resolved these problems by introducing an appropriate definition of a weak solution, inspired by the notion of viscosity solutions, and showed existence and uniqueness.

When the initial hypersurface is mean convex (the mean curvature is non-negative), so are all future ones and the front advances monotonically.  In this case, Evans and Spruck, \cite{ES}, showed that $v(x, t) = u(x) - t$, where $u$ is Lipschitz and satisfies (in the viscosity sense)
\begin{align}	\label{e:arrivalu}
	-1 = |\nabla u|\,\text{div}\left( \frac{\nabla u}{|\nabla u|}\right)\,  .
\end{align}
As the front moves monotonically inwards, it sweeps out the entire domain inside the initial hypersurface.  The function $u$ is the {\emph{arrival time}} since $u(x)$ is the time when the front passes through $x$.
It is defined on the entire compact domain bounded by the initial hypersurface.
Singular points for the flow correspond to critical points for $u$: the flow has a singularity at $x$ at time $u(x)$ if and only if $\nabla u (x) = 0$.

 When  the initial hypersurface is convex,   the flow is smooth except at the point it becomes extinct and Huisken showed that the arrival time is $C^2$, \cite{H1}, \cite{H2}.  In \cite{I1}, \cite{I2},
 Ilmanen gave  an example of a rotationally symmetric mean convex dumbbell in $\RR^3$ for which the arrival time was not $C^2$.  
 There is even more regularity in the plane, where Kohn and Serfaty showed that it is at least $C^3$, \cite{KS}. For $n > 1$, Sesum, \cite{S},  showed that Huisken's result is optimal; namely, she 
 gave examples of convex initial hypersurfaces where the arrival time is not three times differentiable.

 \vskip2mm
 In the next two theorems and   corollary,   $u$ is the arrival time of a mean convex flow in $\RR^{n+1}$ starting from a smooth closed connected  hypersurface.

 \begin{Thm}	\label{t:mcvxRn}
  $u$ is $C^2$ if and only if  both (1) and (2) hold:
  \begin{enumerate}
  \item There is exactly one singular time $T$ (where the flow becomes extinct).
  \item  The singular set $\cS$ is  
a $k$-dimensional closed connected embedded $C^1$ submanifold of  singularities where the blowup is a cylinder $\SS^{n-k} \times \RR^k$ at each point.  
\end{enumerate}
Moreover, $\cS$ is tangent to the $\RR^k$ factor in (2).
\end{Thm}

In general, even if $u$ is not $C^2$, it follows from \cite{CM4} that $\cS$ is contained in a union of $C^1$ submanifolds  with each submanifold   tangent to the axis of the corresponding cylinder at each singular point.{\footnote{The main theorem of \cite{CM4} states the submanifolds are Lipschitz, but the proof shows that they are $C^1$.}}  There are finitely many $(n-1)$-dimensional submanifolds and at most countably many in each lower dimension.  Theorem \ref{t:mcvxRn} gives a much stronger statement when $u$ is $C^2$: there is only one submanifold, it is closed connected and embedded, it lies in one singular time, and $\cS$ fills out the entire submanifold (rather than being a subset of it).

A convex MCF  
gives an example where $u$ is $C^2$ and $\cS$ is a point (i.e., $k=0$), while the marriage ring{\footnote{The marriage ring is a thin mean convex torus
of revolution in $\RR^3$ where the MCF is smooth until   it  becomes extinct along a circle.}}
 gives an example where $u$ is $C^2$ and $\cS$ is a circle of cylindrical singularities.
In contrast, any of the examples of rotationally symmetric surfaces studied in \cite{AAG} has isolated cylindrical singular  points and, thus, is not $C^2$.  

We can restate the theorem in terms of the function $u$ as follows:

  \begin{Cor}	\label{c:mcvxRn}
  $u$ is $C^2$  if and only if  both (1) and (2) hold:
  \begin{enumerate}
  \item There is exactly one critical value $T= \max u$.
  \item  The critical set $\cS$ is    
a $k$-dimensional closed connected embedded $C^1$ submanifold.  At each critical point,      $\Hess_u$ has a $k$-dimensional kernel tangent to the critical set and is $- \frac{1}{n-k}$ times the identity on the orthogonal complement.
\end{enumerate}
\end{Cor}

 \vskip2mm
 The Hessian is always continuous where the flow is smooth.  Thus,  discontinuity of $\Hess_u$ only occurs at  critical points of $u$.  
The next proposition shows that $\Hess_u$ is still continuous at a critical point if we approach it transversely to the kernel $K$ of $\Hess_u$ at the critical point:  $u$ is $C^2$ where the projection $\Pi_{\text{axis}}$ onto $K$ is bounded by the projection $\Pi $ onto $K^{\perp}$.

 \begin{Thm}	\label{t:C2trans}
Suppose that $\nabla u (0) = 0$.
  Given any $C$, there exists $\delta > 0$ so that   $u$ is $C^2$ in the region
\begin{align}	\label{e:region}
 	B_{\delta} \cap \{ x \, | \, \left| \Pi_{\text{axis}} (x) \right|  \leq C  \, \left| \Pi (x) \right|   \} \, .
\end{align}
 \end{Thm}

Thus, any  lack of continuity only occurs along paths   tangent to the kernel of $\Hess_u$.

 \section{ $C^2$ arrival times}
 
%\begin{Thm}	\label{t:notC2}
%Suppose that $u$ is the arrival time for a mean convex flow.  If   $p$ is in the closure of $\{ x \, | \, u(x) > u(p) \}$, then $\Hess_u$ is not continuous at $p$.
%\end{Thm}

In this section, we will prove one direction of the main theorem: If the arrival time is $C^2$, then the flow has the one singular time and the singular set is a closed connected embedded $C^1$ submanifold.

Throughout this section,    $u$ is the arrival time of a mean convex flow in $\RR^{n+1}$ starting from a smooth closed connected  hypersurface.
 
\subsection{The stratification of $\cS$}
When the initial hypersurface is mean convex, then all singularities are cylindrical; see, \cite{W1}, \cite{W2},   \cite{H1}, \cite{HS1}, \cite{HS2}, \cite{HaK}, \cite{An}; cf. \cite{B}, \cite{CM1}.

The singular set $\cS$ is stratified into subsets 
\begin{align}
	\cS_0 \subset \cS_1 \subset \dots \subset \cS_{n-1} = \cS \, , 
\end{align}
where $\cS_k$ consists of all singularities  where the tangent flow splits off a Euclidean factor of dimension {\underline{at most}}
$k$.  In particular, $\cS_k \setminus \cS_{k-1}$ is the set where the blow up is $\RR^k \times \SS^{n-k}$.  
By \cite{CM5}, the Hessian has a  special form at a critical point.  Namely,  if $p \in \cS_k \setminus \cS_{k-1}$, then 
\begin{align}	\label{e:1p2}
	\Hess_u (p) = - \frac{1}{n-k} \, \Pi \, , 
\end{align}
 where $\Pi$ is orthogonal projection onto the orthogonal complement of the $\RR^k$ factor.
If $k\geq 1$, let $\Pi_{\text{axis}}$ denote orthogonal projection onto the $k$-plane tangent to the ``axis''.

It follows from upper semi-continuity of the density that the top strata $\cS \setminus \cS_{n-2}$ is compact.  A priori, it is possible that a sequence of points in one of the lower strata might converge to a point in a higher strata.  However, by  \eqr{e:1p2}, this is impossible when the arrival time is $C^2$:

\begin{Lem}	\label{l:strata}
  If $u$ is $C^2$, then each strata $\cS_k \setminus \cS_{k-1}$ is compact.

\end{Lem}

\begin{Lem}	\label{l:normals}
If $u$ is $C^2$ at $p \in \cS_k \setminus \cS_{k-1}$ with $k\geq 1$ and $q_j$ is a sequence of regular points converging to $p$, then
\begin{align}
	\Pi_{\text{axis}} (\nn (q_j)) \to 0 \, .
\end{align}
\end{Lem}

\begin{proof}
We will argue by contradiction, so suppose instead that there is a sequence $q_j \to p$ with $|\Pi_{\text{axis}} (\nn (q_j))| \geq \delta > 0$.  Since $\SS^n$ is compact, we can pass to a subsequence so that
$\nn (q_j) \to V \in \SS^n$.  In particular, we must have
\begin{align}	\label{e:willcon}
	\left|\Pi_{\text{axis}} (V) \right| \geq \delta > 0 \, .
\end{align}
Using the arrival time equation \eqr{e:arrivalu} at the smooth points $q_j$ and then passing to limits since $u$ is $C^2$, we get that
\begin{align}
	0 =&\lim_{j \to \infty}  \, \left(  1 + \Delta u (q_j) -   \Hess_u (q_j) (\nn (q_j) , \nn (q_j) ) \right) = 1 + \Delta u (p) - \Hess_u (p) (V, V) \notag \\
	&= - \frac{1}{n-k} - \left[ - \frac{1}{n-k} \, \langle \Pi (V) , V \rangle \right] = - \frac{1}{n-k} \, \left|\Pi_{\text{axis}} (V) \right|^2 \, .
\end{align}
 This contradicts \eqr{e:willcon}, giving the lemma.
\end{proof}

The next lemma, which does not assume that $u$ is $C^2$, shows that a plane  orthogonal to the axis of a singularity contains a  point $q$ where $\Pi (\nabla u (q)) = 0$.  

\begin{Lem}	\label{l:Piz}
Suppose that   $\nabla u (0) = 0$ and $\Hess_u (0)$ has   kernel $K$. There exist $\epsilon > 0$ and $C$ so that if $p \in B_{\epsilon} \cap K$, then
there exists $q \in B_{C\, |p|} \cap \left(p+ K^{\perp} \right) $ with $\Pi (\nabla u (q)) = 0$.
\end{Lem}

\begin{proof}
By  the uniqueness of  \cite{CM2}, the flow is cylindrical at time $t = u(0) - \sqrt{\delta}$ in a ball $B_{C' \, \delta}(p)$ for every $\delta \in (0, \epsilon)$ for some $\epsilon > 0$ sufficiently small.   Here $C' $ is a large constant.{\footnote{We can make $C'$ as big as we want at the cost of decreasing $\epsilon$.}}
Thus, since $p \in B_{\epsilon} \cap K$,   the level set $\{ u = u(0) - \sqrt{|p|} \}$ is an approximate cylinder about $K$ in $B_{C'|p|}$. In particular, the intersection 
\begin{align}
	\{ u = u(0) - \sqrt{|p|} \} \cap  \left(p+ K^{\perp} \right)
\end{align}
is close to an $\SS^{n-k+1}$ and, furthermore, $u$ is strictly decreasing at each point of the intersection.  Let $q \in   \left(p+ K^{\perp} \right)$ be the point where $u$ achieves its maximum inside the subset of 
$\left(p+ K^{\perp} \right)$ bounded by $\{ u = u(0) - \sqrt{|p|} \} $.  It follows that $q$ is in the interior and, thus, $\nabla u(q)$ is orthogonal to $K^{\perp}$ as claimed.
 \end{proof}

 \subsection{Local lemma}

In this subsection, we assume that    $u$ is $C^2$.
The key  to Theorem \ref{t:mcvxRn}   is the following local proposition:

\begin{Pro}	\label{p:key}
Suppose that $\nabla u (0) = 0$ and $\Hess_u (0)$ has  kernel $K$.  Then there exists $\epsilon > 0$ so that
$B_{\epsilon} \cap \cS$ is the graph of a $C^1$ map
\begin{align}
	f: \Omega \subset K \to K^{\perp} \, ,
\end{align}
where $\Omega$ is a connected open subset of $K$ containing $0$.  Furthermore, $u$ is constant on $B_{\epsilon} \cap \cS$.
\end{Pro}

  \begin{proof}
  It follows from theorem $2.5$ and corollary $4.5$ in \cite{CM4} that there is some $\delta > 0$ so that $B_{\delta}  \cap \cS$ is 
  {\underline{contained in}} the graph of a $C^1$ map{\footnote{The main theorem of \cite{CM4} states that the map $f$ below is Lipschitz.  However, the regularity of the distribution of $k$-planes implies that it is in fact $C^1$.}}
\begin{align}
	f: \Omega \subset K \to K^{\perp} \, .
\end{align}
Moreover, $B_{\delta}  \cap \cS$ is automatically a (relatively) closed subset of this graph.  To prove the the first part of the proposition, we   show that we can choose some $\epsilon \in (0 , \delta ]$ so $\cS$ fills out the entire graph in $B_{\epsilon}$.  To do this, we must rule out the following possibility:
\begin{itemize}
\item[($\star$)] There is a sequence $p_j \to 0$ of points $p_j \in K$ so that the plane $P_j$ through $p_j$ and parallel to $K^{\perp}$ misses 
$B_{\delta} \cap \cS$.
\end{itemize}
We will show that ($\star$) leads to a contradiction.  Namely, for each $j$, Lemma \ref{l:Piz} gives a point $q_j \in B_{C|p_j|} \cap P_j$ with
\begin{align}	\label{e:qjp}
	\Pi (\nabla u (q_j)) = 0 \, ,
\end{align}
where $\Pi $ is orthogonal projection onto $K^{\perp}$.  Since $\cS$ does not intersect $P_j$, we know that $\nabla u (	q_j) \ne 0$.  Therefore, \eqr{e:qjp} gives that
\begin{align}	\label{e:qjp2}
	\Pi  (\nn (q_j)) = 0 \, .
\end{align}
However, this contradicts Lemma \ref{l:normals} since $q_j \to 0$.  Thus, we get the desired $\epsilon > 0$.  This gives the first part of the proposition.

Next, we must show that this graph is contained in a level set of $u$.  This follows immediately from part (B) of theorem $1.2$ in  \cite{CM4} since any two points in the graph can be connected by a $C^1$ curve in $\cS$.

  \end{proof}
  
  \subsection{Local extinction after singularities}

   In the next lemma, $p \in \cS_k \setminus \cS_{k-1}$ is a singularity of the flow and $K^{\perp}_p$ is the
$n+1 - k$ dimensional plane through $p$ orthogonal to the axis of the singularity.

\begin{Lem}	\label{l:sepa}
There exists $\epsilon > 0$, depending only on $u$ and not on $p$, so that
\begin{itemize}
\item $B_{\epsilon} (p) \cap \{ u > u(p) \}$ does not intersect $K^{\perp}_p$.
\end{itemize}
\end{Lem}

\begin{proof}
By  the uniqueness of  \cite{CM2}, the flow is cylindrical at time $t = u(p) - \sqrt{\delta}$ in a ball $B_{C \, \delta}(p)$ for every $\delta \in (0, \epsilon)$ for some $\epsilon > 0$ sufficiently small.  Here $\epsilon > 0$ depends only on the cylindrical scale and, thus, is uniform in $p$ by
  theorem $3.1$ in \cite{CM4}  because  each strata is  compact by Lemma \ref{l:strata}.

 The intersection of the level set 
$u = u(p) - \sqrt{\delta}$ with $K^{\perp}_p$  is an $(n-k)$ sphere that separates $K^{\perp}_p$ (at least in the ball $B_{\epsilon} (p)$) into an inside containing $p$ and an outside where the flow has recently gone through.  Because the flow is monotone, it can never return to this outside region.  By assumption, these inside regions shrink to $p$ as $\delta \to 0$.
\end{proof}

The next corollary shows that if a critical time can be approached by future regular times, then each critical point at this time is a local maximum.  

 \begin{Cor}	\label{c:extinct}
Suppose that $u$ is $C^2$, $\nabla u (0) = 0$, and there exist $t_i > u(0)$ with $t_i \to u(0)$ and $\nabla u \ne 0 $ on $\{ u = t_i \}$.  Then there exists $\delta > 0$ so that 
\begin{align}
	\sup_{B_{\delta}} u = u (0) \, .
\end{align}
 \end{Cor}

\begin{proof}
Let $\epsilon > 0$ be from Lemma \ref{l:sepa}.    
We will argue by contradiction, so suppose instead
that there is a sequence $p_j \to 0$ with $u(p_j) > u(0)$.  By continuity of $u$, $u(p_j) \to u(0)$.  Thus, after passing to   subsequences for the $p_j$'s and $t_j$'s, we can assume that 
\begin{align}
	u(p_1) > t_1 > u(p_2) > t_2 > \dots \to u(0) \, .
\end{align}

 Suppose that $i$ is large so that $|p_i| < \epsilon$.  Since $u$ is continuous and  $u(p_i)  > t_i > u(0) $, the  line segment from $0$ to $p_i$ intersects $\{ u= t_i \}$.  Thus, we can
choose   $q_i$ with
\begin{align}	\label{e:closew}
	 \left| \Pi_{\text{axis}} (q_i) \right|^2 = \min \left\{ \left| \Pi_{\text{axis}} (q) \right|^2 \, | \, q \in B_{\epsilon} {\text{ and }} u(q) = t_i \right\} \leq |p_i|^2  \, .
\end{align}
This has two consequences:
\begin{align}
	 \left|  q_i \right|^2 & \to 0 \, , \label{e:qit0} \\
 	\Pi (\nn (q_i)) &= 0 \, . \label{e:nablaud}
\end{align}
To prove \eqr{e:qit0}, use \eqr{e:closew} to get that
  $ \left| \Pi_{\text{axis}} (q_i) \right|^2 \to 0$  and then use that 
  the support of the flow for $u> u(0)$ must be close to $K$ near $0$ (by theorem $3.1$ in \cite{CM4}).  
  
  To see \eqr{e:nablaud}, let $h: u^{-1} (t_i) \to \RR$ be given by $h(x) = \left|   \Pi_{\text{axis}} (x) \right|^2$, so that
  \begin{align}	
  	\frac{1}{2} \, \nabla_x h =        \Pi_{\text{axis}} (x)-   \langle     \Pi_{\text{axis}} (x)  , \nn (x) \rangle \, \nn (x)  \,  .
\end{align}
Since $q_i$ is a minimum of $h$, we get that $\nabla_{q_i} h = 0$ and, therefore,
  \begin{align}
  	 \Pi_{\text{axis}} (q_i) =    \langle     \Pi_{\text{axis}} (q_i)  , \nn (q_i) \rangle \, \nn (q_i) \, .
\end{align}
It follows that $ \Pi_{\text{axis}} (q_i) = \pm \, \left|  \Pi_{\text{axis}} (q_i) \right| \, \nn (q_i)$.  This implies that  
\begin{align}
	 \Pi_{\text{axis}} (q_i) = 0 {\text{ or }}  \Pi  (\nn (q_i)) = 0 \, .
\end{align}  	
Lemma \ref{l:sepa} rules out the first possibility, so we get  \eqr{e:nablaud}.
  
On the other hand, \eqr{e:qit0} allows us to apply 
 Lemma \ref{l:normals} to get that
 \begin{align}	\label{e:piax}
 	\Pi_{\text{axis}} (\nn (q_i)) \to 0 \, .
 \end{align}
This  contradicts \eqr{e:nablaud}, completing the proof. 
\end{proof}

\subsection{Proofs of the main results}

We will prove one direction of Theorem \ref{t:mcvxRn} in the following proposition.

\begin{Pro}	\label{p:mcvxRn}
 If $u$ is $C^2$, then
  \begin{enumerate}
  \item There is exactly one singular time $T$ (where the flow becomes extinct).
  \item  The singular set $\cS$ is  
a $k$-dimensional closed connected embedded $C^1$ submanifold of  singularities where the blowup is a cylinder $\SS^{n-k} \times \RR^k$ at each point.  
\end{enumerate}
Moreover, $\cS$ is tangent to the $\RR^k$ factor in (2).
\end{Pro}

  \begin{proof} 
  Fix a point $p \in \cS$.  Let $k$ be the dimension of the kernel of $\Hess_u (p)$, so $p$ is cylindrical of type $\SS^{n-k} \times \RR^k$.  
   Let $\cS_p$ be the component of $\cS$ containing $p$; note that each point in $\cS_p$ must also be cylindrical of type $\SS^{n-k} \times \RR^k$ by
  Lemma  \ref{l:strata}.   Given $q \in \cS_p$, let $K^{\perp}_q$ be the $k$-dimensional kernel of $\Hess_u (q)$.
  
  Proposition \ref{p:key} implies that each point $q$ in $\cS_p$ has an $\epsilon_q > 0$ so that
  \begin{itemize}
  	  \item $B_{\epsilon_q}(q) \cap \cS$  is given as a $C^1$ graph over $K^{\perp}_q$.
	  \item $u$ is constant on this graph.
\end{itemize}
Since $\cS_p$ is compact and connected, it follows that $\cS_p$ is a closed connected embedded $C^1$ $k$-dimensional submanifold and $u \equiv u(p)$ on $\cS_p$. 

Since $\cS$ is compact, we conclude that $\cS$ is given as a finite collection of disjoint embedded  $C^1$   closed submanifolds 
\begin{align}
	\cS = \cup_{j=1}^N \cS_{p_j} {\text{ with }} u (\cS_{p_j}) \equiv u (p_j) \, .
\end{align}
  
  Let $T$ be the first singular time.  In the remainder of the proof, we will show that
  \begin{enumerate}
  \item[(A)] $T$ is also the extinction time and, thus,  the only singular time.
  \item[(B)] $\cS$ has only one component. 
 \end{enumerate}
   Let $\cS_T = \cS \cap \{ u = T \}$ be the union of the $\cS_{p_j}$'s where $u(p_j) = T$.  Note that $\cS_T$ is compact and there exists $\kappa > 0$ so that 
  \begin{align}
  	\cS \cap \{ T < u < T+ \kappa \}  = \emptyset 
\end{align}
since there are only finitely many singular times.  Thus, Corollary \ref{c:extinct} gives $\delta > 0$ so that
\begin{align}	\label{e:loce}
	\sup_{ T_{\delta} (\cS_T)} \, u = T \, , 
\end{align}
where $T_{\delta} (\cS_T)$ is the $\delta$-tubular neighborhood of $\cS_T$.     

We can now prove (A) by contradiction. Namely, if (A) does not hold, then \eqr{e:loce} and the monotonicity of the flow imply that $\{ u = t \}$ intersects both inside and outside of $T_{\delta/2} (\cS_T)$ for $t < T$.
   Since the initial hypersurface is connected and the flow is smooth before $u=T$, we know that $\{  u = t \}$ is connected for each $t< T$.  Thus, we get a sequence of points $z_j \in \partial T_{\delta/2} (\cS_T)$ 
   with $u(z_j ) < T$ and $u(z_j) \to T$.  By compactness, a subseqence of the $z_j$'s converges to $z \in \partial T_{\delta/2} (\cS_T)$.  Continuity of $u$ implies that $u(z) = T$ and, thus, \eqr{e:loce} implies that 
   $z$ is a local maximum for $u$ and $\nabla u (z) = 0$.  This contradicts that $z \in \partial T_{\delta/2} (\cS_T)$ is not a critical point, giving (A).
   
   Now that we know that every point in $\{ u = T \}$ is a critical point, the same argument that we used for (A) implies that $\cS = \cS_T$ is connected.  This gives (B), completing the proof.
     \end{proof}

\section{The arrival time is $C^2$ away from the axis}

Throughout this section, $u$ will be the arrival time for a mean convex flow in $\RR^{n+1}$ starting from a smooth closed mean convex hypersurface.  By \cite{CM5}, $u$ is twice differentiable everywhere with bounded $\Hess_u$ and is smooth away from the singular set where $\nabla u = 0$.

\begin{proof}[Proof of Theorem \ref{t:C2trans}]
It follows from \cite{CM4} that the region in \eqr{e:region} intersects the singular set only at $0$ for $\delta > 0$ small enough.  Thus, by \cite{CM5}, we need only show that any sequence $q_j \to 0$ in \eqr{e:region} must have $\Hess_u (q_j) \to \Hess_u (0)$.  Furthermore, by lemma $2.11$ of \cite{CM5}, $u(x) \leq u (0)$ in the region \eqr{e:region} with equality only for $x=0$.

If $e_1 ,\dots e_n$ is an orthonormal frame for the level sets of $u$, then 
\begin{align}
	\Hess_u (e_i , e_j) &= \frac{A(e_i , e_j)}{H} \, , \\
	\Hess_u (\nn , \nn) &= \nabla_{\nn}|\nabla u| = - \frac{\partial_t H}{H^3} = - \frac{(\Delta + |A|^2) H}{H^3} \, , \\
	\Hess_u (e_i , \nn) &= \nabla_{e_i} |\nabla u| = - \frac{H_i}{H^2} \, .
\end{align}
%Given a point $x$, let $x^P = P(x)$ and $x^{\perp} = P^{\perp} (x)$.

In the region \eqr{e:region}, the uniqueness of \cite{CM2} gives that the rescaled level set flow converges to  cylinders with axis $K$.  If we let $\rho$ denotes the distance to $K$, then 
\begin{align}
	\frac{\nabla u}{|\nabla u|} \to \partial_{\rho} {\text{ and }}
	\frac{1}{H \rho} = \frac{ |\nabla u|}{\rho} \to 1 \, , \\
	-\frac{A}{H} \to  \frac{1}{n-k}  \, \Pi  {\text{ restricted to the tangent space}}, \\
	\frac{ \left| \nabla \frac{A}{H} \right|}{H}   \, ,  
	\frac{ | \nabla H |}{H^2} {\text{ and }} 
	\frac{\left| \Delta H \right|}{H^3} \to 0 \, .
\end{align}
The first three claims are immediate from the uniqueness of the blow up.  The last three claims follow from the smooth convergence of the rescaled level sets to the cylinder (where each of these quantities is zero); the powers of $H$ are the appropriate scaling factors.

Combining these facts shows that $\Hess_u$ is continuous in this conical region.

\end{proof}

  \begin{proof}[Proof of Theorem \ref{t:mcvxRn}]
  One direction is given by Proposition   \ref{p:mcvxRn}.  We will suppose therefore that (1) and (2) hold and show that $u$ must be $C^2$.  By \cite{CM5}, $u$ is twice differentiable everywhere and smooth away from the singular set $\cS$.  Thus, we must show that $\Hess_u$ is continuous at each point of $\cS$.
  
Using the   form of the Hessian, it follows that  if $p , \tilde{p} \in \cS$, then 
\begin{align}	\label{e:huclo}
	\left| \Hess_u (p) - \Hess_u (\tilde{p}) \right| \leq C \, \dist (T_p \cS , T_{\tilde{p}} \cS ) \, .
\end{align}
Fix a point $p \in \cS$ and let $q_j \to p$ be any sequence.  We must show that $\Hess_u (q_j) \to \Hess_u (p)$.  For each $j$, let $p_j$ be a closest point in $\cS$ to $q_j$.  It follows that
\begin{itemize}
\item $|p_j - q_j| \leq |p- q_j| \to 0$.
\item $\langle (p_j - q_j ) , T_{p_j} \cS) \rangle = 0$.
\end{itemize}
The second property allows us to apply Theorem \ref{t:C2trans} to get that
\begin{align}
	\left| \Hess_u (p_j) - \Hess_u (q_j) \right| \to 0 \, .
\end{align}
Finally, since $\cS$ is $C^1$ and $p_j \to p$, \eqr{e:huclo} implies that 
\begin{align}
	\left| \Hess_u (p_j) - \Hess_u (p) \right| \to 0 \, .
\end{align}

  \end{proof}

\end{document}